\documentclass[11pt]{amsart}

\usepackage{amsmath}
\usepackage{amssymb}
\usepackage{mathrsfs}
\usepackage{comment}
\usepackage{color}
\usepackage[colorlinks,citecolor=blue,urlcolor=black,linkcolor=black]{hyperref}

\newtheorem{theorem}{Theorem}[section]

\newtheorem{proposition}[theorem]{Proposition}
\newtheorem{corollary}[theorem]{Corollary}
\newtheorem{obs}[theorem]{Observation}

\theoremstyle{definition}
\newtheorem{definition}[theorem]{Definition}

\newtheorem{question}[theorem]{Question}

\theoremstyle{remark}

\newcount\skewfactor
\def\mathunderaccent#1#2 {\let\theaccent#1\skewfactor#2
\mathpalette\putaccentunder}
\def\putaccentunder#1#2{\oalign{$#1#2$\crcr\hidewidth
\vbox to.2ex{\hbox{$#1\skew\skewfactor\theaccent{}$}\vss}\hidewidth}}

\def\smallbox#1{\leavevmode\thinspace\hbox{\vrule\vtop{\vbox
   {\hrule\kern1pt\hbox{\vphantom{\tt/}\thinspace{\tt#1}\thinspace}}
   \kern1pt\hrule}\vrule}\thinspace}

\newcommand{\cf}{{\rm cf}}

\def\qedref#1{$\qed_{\reforiginal{#1}}$}

\title{Splendid extensions}
\author{Shimon Garti}
\address{Einstein Institute of Mathematics,
 The Hebrew University of Jerusalem,
 Jerusalem 9190401, Israel}
\email{shimon.garty@mail.huji.ac.il}

\thanks{}

\begin{document}
\let\labeloriginal\label
\let\reforiginal\ref

\begin{abstract}
We show that one can force the non-existence of $\kappa^+$-splendid models, where $\kappa$ is an infinite cardinal.
We prove a similar result in \textsf{ZFC} under the assumption of an ineffable cardinal.
\end{abstract}

\maketitle

\newpage

\section{Introduction}

Let $N$ be a model, and let $I$ be a cut of $N$.
The collection of $X\subseteq{I}$ for which there is $Y\in{\rm Def}(N)$ so that $X=Y\cap{I}$ is denoted by ${\rm Cod}(N/I)$.
Usually, we deal with ${\rm Cod}(N/M)$ where both $M$ and $N$ are models and $N$ is an extension of $M$.\footnote{Models in this paper are models of \textsf{PA}, and extensions are proper extensions, unless otherwise stated.}
The following comes from \cite[Corollary 8.2.5]{MR2250469}, and it is fundamental for the results in the current paper.

\begin{proposition}
  \label{propks} Every countable recursively saturated model $M$ has two countable recursively saturated elementary end-extensions $M_0$ and $M_1$ such that ${\rm Cod}(M_0/M)\cap{\rm Cod}(M_1/M)={\rm Def}(M)$.
\end{proposition}

It should be noted that if $M\prec_{e}N$ and $N$ is recursively saturated then $N$ \emph{is not} conservative over $M$, see \cite[Proposition 10.1.2]{MR2250469}.
Thus, recursive saturation contradicts conservativity if the extension is elementary.
The above statement shows that one can get something that resembles conservativity, by forking over $M$ with two recursively saturated models $M_0,M_1$, each of which may add some new definable sets but together they form a parallel to a conservative extension.

An essential component of the proof is the Baire category theorem, in which one intersects countably many dense sets in some relevant topology.
In particular, it is not clear whether the countability of the pertinent models is indispensable in the above proposition.
The main result of the paper reads as follows.

\begin{theorem}
  \label{thmmt} Let $\kappa$ be a successor cardinal.
  It is consistent that every recursively saturated model of cardinality $\kappa$ extends to a model $M$, where $|M|=\kappa$, and if $M_0,M_1$ are recursively saturated elementary end-extensions of $M$ of size $\kappa$ then ${\rm Cod}(M_0/M)\cap{\rm Cod}(M_1/M)\neq{\rm Def}(M)$.
\end{theorem}

We assume that there are large cardinals in the ground model, in order to force this result, though we do not know whether it has consistency strength.
Moreover, under the assumption of an ineffable cardinal $\kappa$ we prove, without forcing, that there are such models of cardinality less than $\kappa$.

The paper contains three additional sections.
The first section collects basic facts and definitions.
In the second section we prove the main theorem with respect to $\aleph_1$ and larger successor cardinals.
In the third section we obtain a similar result below large cardinals, and pose some natural open problems.
Our notation is (hopefully) standard, and we refer to \cite{MR1098499} and \cite{MR2250469} for background in models of Peano arithmetic.

\newpage

\section{Preliminaries}

Let us commence with the central definition of the paper.

\begin{definition}
  \label{defsplendid} Let $\kappa$ be an infinite cardinal and let $M$ be a model of cardinality $\kappa$.
  A pair of models $(M_0,M_1)$ is a splendid pair over $M$ iff $|M_0|=|M_1|=\kappa$, both are recursively saturated elementary end extensions of $M$, and ${\rm Cod}(M_0/M)\cap{\rm Cod}(M_1/M)={\rm Def}(M)$.
\end{definition}

As mentioned in the introduction, if $\kappa=\aleph_0$ then every model $M$ of size $\kappa$ admits a splendid pair.
The situation might be different for $\kappa=\aleph_1$, or for higher cardinalities.
In order to prove this, we need the concept of Kaufmann models.

Let $M$ be a model of size $\kappa$.
One says that $M$ is $\kappa$-like if $|M|=\kappa$ but $|a_M|<\kappa$ for every $a\in{M}$, where $a_M=\{b\in{M}\mid M\models b\leq{a}\}$.
If $\kappa=\aleph_0$ then the only $\kappa$-like model of \textsf{PA} is the standard model, but if $\kappa>\aleph_0$ then one can construct many $\kappa$-like models of \textsf{PA}.
The common way to build such models is by constructing an increasing continuous chain of elementary end-extensions of a given model $M$, using the MacDowell-Specker theorem from \cite{MR152447}.

A model $M$ is \emph{recursively saturated} if every recursive type over $M$ is realized in $M$.
We assume that there is a fixed G\"odel numbering of first-order formulae in our language, and $p$ is a recursive type if the set of natural numbers corresponding to the formulas in $p$ is a recursive set.
Note that if $M$ is non-standard then every $\Sigma_n$-recursive type is realized in $M$.
Being recursively saturated is slightly stronger, as a recursive type $p$ might contain formulas with unbounded degree in the arithmetic hierarchy.
By and large, recursive saturation indicates that the model is rich in some sense.

A subset $X$ of $M$ is called \emph{a class} if $X\cap{a}$ is definable in $M$ for every $a\in{M}$.
It is fairly possible that $X$ itself is not definable in $M$.
In fact, the following statement holds.

\begin{obs}
  \label{obsmanyclasses} Let $M$ be a model of \textsf{PA} for the language $\mathcal{L}\supseteq\mathcal{L}_{\textsf{PA}}$, where $|M|=\lambda, |\mathcal{L}|\leq\lambda$ and $\cf(\lambda)=\omega$.
  Then there is a class $X\subseteq{M}$ so that $X\notin{\rm Def}(M)$.
\end{obs}

\par\noindent\emph{Proof}. \newline
If $X\subseteq{M}, {\rm otp}(X)=\omega$ and $X$ is unbounded in $M$ then $X$ is a class of $M$ since every finite subset of $M$ is definable.
There are $\lambda^\omega\geq\lambda^+$ such classes, but only $\lambda$-many definable classes since $|M|,|\mathcal{L}|\leq\lambda$.

\hfill \qedref{obsmanyclasses}

A model $M$ is \emph{rather classless} if every class of $M$ is definable.\footnote{The term rather classless was coined by Kaufmann in \cite{MR476498}.}
Being rather classless means that $M$ is poor in some sense.
The seminal paper of Kaufmann, \cite{MR476498}, investigates the possibility that a model $M$ is both rich and poor.\footnote{See the story of the poor man's ewe lamb, \cite[Chapter 12]{samuel}.}

\begin{definition}
  \label{defkaufmannmodel} Let $M$ be a model of \textsf{PA} and let $\kappa$ be an infinite cardinal.
  One says that $M$ is $\kappa$-Kaufmann iff $M$ enjoys the following properties:
  \begin{enumerate}
    \item [$(\aleph)$] $M$ is $\kappa$-like.
    \item [$(\beth)$] $M$ is recursively saturated.
    \item [$(\gimel)$] $M$ is rather classless.
  \end{enumerate}
\end{definition}

By Observation \ref{obsmanyclasses}, if $\kappa=\aleph_0$ or $\kappa>\cf(\kappa)=\omega$ then there are no $\kappa$-Kaufmann models.
On the other hand, if $\kappa>\cf(\kappa)>\omega$ then there exist $\kappa$-Kaufmann models as proved by Schmerl in \cite{MR619874}.
We focus, therefore, on regular uncountable cardinals.

Kaufmann proved in \cite{MR476498} that if $\Diamond_{\aleph_1}$ holds then there exists an $\aleph_1$-Kaufmann model.
Recall that $(A_\alpha\mid\alpha\in\omega_1)$ is a diamond sequence if $\{\alpha\in\omega_1\mid A\cap\alpha=A_\alpha\}$ is a stationary subset of $\omega_1$ whenever $A\subseteq\omega_1$.
One says that $\Diamond_{\aleph_1}$ holds if there exists a diamond sequence on $\omega_1$.
Clearly, $\Diamond_{\aleph_1}$ implies \textsf{CH} and hence this principle is not provable in \textsf{ZFC}.
However, $\aleph_1$-Kaufmann models do exist in \textsf{ZFC}.
By an absoluteness result of Shelah from \cite{MR501098}, Kaufmann's theorem actually holds in \textsf{ZFC}.

The situation is different if $\kappa>\aleph_1$.
Schmerl showed in \cite[Theorem 3]{MR619874} that if $\kappa=\cf(\kappa)>\aleph_0$ and there exists a $\kappa$-Kaufmann model then there is a $\kappa$-Aronszajn tree.
The existence of $\aleph_1$-Aronszajn trees is a consequence of \textsf{ZFC}, see \cite{MR3533035}, but already at the second uncountable cardinal one can force the non-existence of $\aleph_2$-Aronszajn trees.\footnote{Occasionally, one says that $\aleph_2$ has the tree property; this means that every tree of height $\omega_2$ and levels of size at most $\aleph_1$ has a cofinal branch. This terminology is common at every uncountable cardinal.}
This result is due to Mitchell, \cite{MR313057}, and in this model there are no $\aleph_2$-Kaufmann models.

The main purpose of the current paper is to show that consistently there are models of cardinality $\kappa^+$ with no splendid extension.
This statement is forceable at every successor cardinal.
The argument proceeds as follows.
Let $M$ be a model of size $\kappa^+$.
Build a tree in which $M$ is the root, and at every successor stage there are two models of cardinality $\kappa^+$ that form a splendid extension over the previous model.
At limit stages take unions.
If this construction is possible for $\kappa^{++}$-many stages, and if the combinatorial principle $\Phi_{\kappa^{++}}$ holds,\footnote{This principle is called \emph{weak diamond}, to be defined below.} then there is an $\kappa^{++}$-Kaufmann model.
Therefore, if one forces the non-existence of $\kappa^{++}$-Kaufmann models and $\Phi_{\kappa^{++}}$ holds, then the above construction is impossible.
Namely, at some stage there will be a model $M$ of size $\kappa^+$ with no splendid extension.
Remark that in this universe there are many models with no splendid extension.
In fact, every model of size $\kappa^+$ extends to a model of the same size with no splendid extension.

To round out the picture, let us recall the definition of the Devlin-Shelah weak diamond.
This principle was defined in \cite{MR469756} with respect to $\aleph_1$.
We phrase here the straightforward generalization to $\kappa$.

\begin{definition}
  \label{defweakdiamond} Let $\kappa$ be regular and uncountable.
  The weak diamond at $\kappa$ is the following statement: for every coloring $c:{}^{<\kappa}2\rightarrow\{0,1\}$ there exists a weak diamond function $g\in{}^\kappa{2}$ such that for every $f\in{}^\kappa{2}$ the set $\{\alpha\in\kappa\mid c(f\upharpoonright\alpha)=g(\alpha)\}$ is a stationary subset of $\kappa$.
\end{definition}

The weak diamond on $\kappa$ is denoted by $\Phi_\kappa$.
It has been proved in \cite{MR469756} that if $2^{\aleph_0}<2^{\aleph_1}$ then $\Phi_{\aleph_1}$ holds.
The converse is true as well, as indicated by U. Abraham.
More generally, if $2^\kappa<2^{\kappa^+}$ then $\Phi_{\kappa^+}$ holds, where the proof is identical with the case of $\aleph_1$.
Likewise, $\Phi_{\kappa^+}$ implies $2^\kappa<2^{\kappa^+}$, see \cite[Proposition 1.2]{MR3604115}.

As mentioned above, Kaufmann used diamond at $\aleph_1$ in his construction.
An alternative proof with weak diamond at $\aleph_1$ is suggested in \cite{MR2250469}.
In the next section we show that this construction works equally well at every uncountable regular cardinal, provided there are appropriate splendid extensions.
Combining this fact with the ability to force models with the tree property at many regular uncountable cardinals, we derive our main theorem.

\newpage

\section{Weak diamond and splendid extensions}

Let $\kappa$ be an infinite cardinal and let $M$ be a model of size $\kappa$.
MacDowell and Specker proved in \cite{MR152447} that there exists a model $N$ of size $\kappa$ so that $M\prec_e{N}$.\footnote{By $M\subseteq_e{N}$ we mean that $N$ is an end-extension of $M$.}
A slightly stronger statement has been proved by Gaifman in \cite{MR406791}.
Call $N$ \emph{a conservative} extension of $M$ if ${\rm Cod}(N/M)={\rm Def}(M)$.
Gaifman showed that every model $M$ extends to an elementary conservative extension $N$ of the same cardinality.
It is easy to see that a conservative elementary extension is an end extension.

The model of Gaifman must be \emph{minimal} over $M$, in some sense.
In particular, it cannot be recursively saturated.
However, if $|M|=\aleph_0$ then one can obtain a splendid pair over $M$, which is quite close to being conservative, see Theorem \ref{propks}.
The situation is different at uncountable cardinals.
The following will be useful for the proof of the main theorem.

\begin{definition}
  \label{defamenable} Let $\kappa$ be a regular cardinal.
  We shall say that $\kappa$ is recursively amenable if every recursively saturated model $M$ of size $\kappa$ has a splendid extension of the same cardinality.
\end{definition}

We can state now the main result:

\begin{theorem}
  \label{thmtreepropertyandwd} Assume that:
  \begin{enumerate}
    \item [$(\aleph)$] $\kappa^+$ has the tree property.
    \item [$(\beth)$] $2^\kappa<2^{\kappa^+}$.
  \end{enumerate}
  Then $\kappa$ is not recursively amenable. In particular, there is a model $M$ of size $\kappa$ with no splendid extension.
\end{theorem}

\par\noindent\emph{Proof}. \newline
Using a pairing function from ${}^{<\kappa^+}2$ to ${}^{<\kappa^+}2\times{}^{<\kappa^+}2$ we may assume that for every $c:{}^{<\kappa^+}2\times{}^{<\kappa^+}2\rightarrow\{0,1\}$ there exists a weak diamond function $g\in{}^\kappa{2}$ such that for every $(f,h)\in{}^\kappa{2}$ the set $\{\alpha\in\kappa\mid c(f\upharpoonright\alpha,h\upharpoonright\alpha)=g(\alpha)\}$ is a stationary subset of $\kappa$.
The existence of $g$ follows from $(\beth)$.

Let $M$ be a recursively saturated model of size $\kappa$.
Assume towards contradiction that $\kappa$ is recursively amenable.
Based on the assumption toward contradiction, we intend to build a $\kappa^+$-Kaufmann model that extends $M$.
To this end, we define a model $M_\eta$ for every $\eta\in{}^{<\kappa^+}2$ as follows.
We let $M_{\langle\rangle}=M$.
Lest $\ell{g}(\eta)$ is a limit ordinal we define $M_\eta=\bigcup_{\zeta\triangleleft\eta}M_\zeta$.
For the successor step, assume that $M_\eta$ is already defined.
Let $M_{\eta^\smallfrown 0}$ and $M_{\eta^\smallfrown 1}$ be recursively saturated elementary end extensions of $M_\eta$ such that ${\rm Cod}(M_{\eta^\smallfrown 0}/M_\eta)\cap{\rm Cod}(M_{\eta^\smallfrown 1}/M_\eta)={\rm Def}(M_\eta)$.
This is possible since $\kappa$ is recursively amenable.
We also assume, without loss of generality, that if $\ell{g}(\eta)$ is a limit ordinal $\delta$ then the universe of $M_\eta$ is $\delta$.\footnote{This can be arranged at a club of $\delta\in\kappa$, so without loss of generality it holds everywhere.}

Suppose that $\alpha\in\kappa^+, \eta\in{}^\alpha{2}$ and $\zeta\in{}^\alpha{2}$.
Let $X_\zeta$ be the set $\{\beta\in\alpha\mid\zeta(\beta)=1\}$.
Define $c(\zeta,\eta)=0$ if $X_\zeta\in{\rm Cod}(M_{\eta^\smallfrown 1}/M_\eta)-{\rm Def}(M_\eta)$ and $c(\zeta,\eta)=1$ if $X_\zeta\in{\rm Cod}(M_{\eta^\smallfrown 0}/M_\eta)-{\rm Def}(M_\eta)$.
In all other cases let $c(\zeta,\eta)=0$.

Let $g\in{}^{\kappa^+}2$ be a weak diamond function for $c$.
Our desired model is $N=\bigcup_{\alpha\in\kappa^+}M_{g\upharpoonright\alpha}$.
The cardinality of $N$ is $\kappa^+$ since we add a new element at every successor step along the branch $g\upharpoonright\alpha$.
On the other hand, if $a\in{N}$ and $\gamma+1$ is the first ordinal for which $a\in M_{\gamma+1}$ then $a_N\subseteq M_{\gamma+1}$ and hence $|a_N|\leq\kappa$.
We conclude, therefore, that $N$ is $\kappa^+$-like.

If $p(\bar{x},\bar{b})$ is a recursive type in $N$ then $\bar{b}\in M_{g\upharpoonright\alpha}$ for some $\alpha\in\kappa^+$, so $p(\bar{x},\bar{b})$ is a recursive type in $M_{g\upharpoonright\alpha}$ and hence realized in $M_{g\upharpoonright\alpha}$.
Since $M_{g\upharpoonright\alpha}\prec{N}$, $p(\bar{x},\bar{b})$ is realized in $N$ as well, thus $N$ is recursively saturated.

Suppose that $X\subseteq{N}$ is a class, and let $f$ be the characteristic function of $X$.
For every $\alpha\in\kappa^+$ let $X_\eta=M_\eta\cap{X}$, where $\eta=g\upharpoonright\alpha$.
Observe that $(M_\eta,X_\eta)\prec(N,X)$ for a club of $\alpha\in\kappa^+$.
Let $\alpha\in\kappa^+$ be a limit ordinal for which $(M_\eta,X_\eta)\prec(N,X)$ and $c(f\upharpoonright\alpha,g\upharpoonright\alpha)=g(\alpha)$.
Denote $f\upharpoonright\alpha$ by $\zeta$ and $g\upharpoonright\alpha$ by $\eta$, so $c(\zeta,\eta)=g(\alpha)$.

We claim that $X_\eta$ is definable in $M_\eta$.
For suppose not.
By the construction, there is $\ell\in\{0,1\}$ so that $X_\eta\in{\rm Cod}(M_{\eta^\smallfrown\ell}/M_\eta)-{\rm Def}(M_\eta)$.
It follows that $g(\alpha)=\ell$.
To see this, observe first that $N$ extends $M_{g(\alpha)}$ by definition.
Now $X_\eta\in{\rm Def}(N)$ and hence $X_\eta\in{\rm Def}(M_{g\upharpoonright\delta})$ for some $\delta\in\kappa^+$.
Necessarily $M_{g\upharpoonright\delta}$ extends $M_{\eta^\smallfrown\ell}$, since $X_\eta\in{\rm Cod}(M_{\eta^\smallfrown\ell}/M_\eta)$.
The latter is true since $M_{g\upharpoonright\delta}$ cannot extend $M_{\eta^\smallfrown 1-\ell}$ by conservativity, bearing in mind that $X_\eta\notin{\rm Cod}(M_{\eta^\smallfrown 1-\ell}/M_\eta)$.
Together, $g(\alpha)=\ell$.
On the other hand, $c(\zeta,\eta)=1-\ell$ since $X_\eta\in{\rm Cod}(M_{\eta^\smallfrown\ell}/M_\eta)-{\rm Def}(M_\eta)$, a contradiction to $(\aleph)$.

We conclude, therefore, that $X_\eta\in{\rm Def}(M_\eta)$.
By the choice of $\alpha$ we know that $(M_\eta,X_\eta)\prec(N,X)$, and hence $X\in{\rm Def}(N)$.
This means that $N$ is rather classless.
Having established all three desired properties we infer that $N$ is a $\kappa^+$-Kaufmann model.
By \cite[Theorem 3]{MR619874}, there exists a $\kappa^+$-Aronszajn tree, a contradiction.

It follows that $\kappa$ is not recursively amenable.
In particular, if one begins with any model $M$ of size $\kappa$ and builds a tree of extensions as described, then at some point there will be a model of size $\kappa$ with no splendid extension.

\hfill \qedref{thmtreepropertyandwd}

From the main theorem of this section we can derive the following conclusion.

\begin{corollary}
  \label{cornegsplendid} Let $\kappa$ be a successor of a regular cardinal and assume there is a weakly compact cardinal above $\kappa$.
  Then it is consistent that there is a model $M$ of size $\kappa$ with no splendid extension.
  Moreover, every model $M'$ of cardinality $\kappa$ extends to a model $M$ of the same cardinality with no splendid extension.
\end{corollary}

\par\noindent\emph{Proof}. \newline
Using Mitchell forcing from \cite{MR313057} one obtains a universe in which there are no $\kappa^+$-Aronszajn trees.
In this model, $2^\kappa=\kappa^+<2^{\kappa^+}$ and hence $\Phi_{\kappa^+}$ holds.
It follows from Theorem \ref{thmtreepropertyandwd} that there is a model $M$ of size $\kappa$ with no splendid extension, as sought.
For the additional statement of the corollary, if $|M'|=\kappa$ then the splendid tree over $M'$ cannot be of height $\kappa^+$.
Hence at some stage, $M'\subseteq_{e}M$ where $|M|=\kappa$ and $M$ has no splendid extensions.

\hfill \qedref{cornegsplendid}

The above corollary applies, in particular, to $\kappa=\aleph_1$.
However, the case of $\aleph_1$ is somewhat less interesting.
The reason is that if $M$ is a Kaufmann model then there is no recursively saturated elementary end-extension of $M$ at all, see \cite[Corollary 10.1.6]{MR2250469}, so the statement is vacuous for these models.
It is still interesting since it says that \emph{every} model of size $\aleph_1$ extends to a model with no splendid extension, but it becomes much more interesting if there are no Kaufmann models at some cardinality, and this is not the case with $\aleph_1$ of course.
Thus the full strength of our theorem is demonstrated by the following:

\begin{theorem}
  \label{thmdoubletreeproperty} Assume there exists a supercompact cardinal and a weakly compact cardinal above it.
  Then one can force the following two statements simultaneously:
  \begin{enumerate}
    \item [$(\aleph)$] There are no Kaufmann models of size $\aleph_2$.
    \item [$(\beth)$] Every model of size $\aleph_2$ extends to a model of the same size with no splendid extension.
  \end{enumerate}
\end{theorem}

\par\noindent\emph{Proof}. \newline
Force the tree property at $\aleph_2$ and $\aleph_3$ together, by \cite{MR717829}.
In the generic extension, there are no $\aleph_2$-Kaufmann models since $\aleph_2$ has the tree property.
Likewise, $\Phi_{\aleph_3}$ and the tree property at $\aleph_3$ hold, thus every model of size $\aleph_2$ extends to a model of the same size with no splendid extension.

\hfill \qedref{thmdoubletreeproperty}

The above argument is applicable at every pair of cardinals of this form.
It is also applicable at successors of singular cardinals by \cite{MR1420265}, but for this forcing construction one needs stronger assumptions of large cardinals in the ground model.

Further progress about the tree property yielded longer intervals of successor cardinals with the tree property.
By \cite{MR1492784} or \cite{MR3224975}, it is consistent that $\aleph_n$ has the tree property for every $n\geq{2}$.
Moreover, in these forcing constructions, $2^{\aleph_n}=\aleph_{n+2}$ for every $n\geq{1}$.
Thus weak diamond holds\footnote{But notice that diamond fails (unlike the parallel situation in Mitchell's model).} along the $\aleph_n$'s for $n\geq{2}$.
It follows that there are models of size $\aleph_n$ with no splendid extension at every $\aleph_n$ where $n\geq{2}$.

\newpage

\section{Large cardinals}

The results in the previous section were obtained by forcing, and thus the existence of models with no splendid extension is only a consistency result.
In this section we show that something similar holds without forcing if there are sufficiently strong concepts of large cardinals at our disposal.

Our strategy would be similar: we wish to combine weak diamonds with the tree property.
Observe that weak diamond holds, in \textsf{ZFC}, at a proper class of cardinals, as shown in \cite{MR4611828}.
This is not the case, however, with respect to the tree property.
But if $\kappa$ is a sufficiently large cardinal then (unlike the case of successor cardinals) one regains the desired properties.

\begin{definition}
  \label{defsubtle} A cardinal $\kappa$ is ineffable if every coloring $c:[\kappa]^2\rightarrow{2}$ admits a monochromatic stationary subset $S$ of $\kappa$.
\end{definition}

An ineffable cardinal is weakly compact, and hence satisfies the tree property.
Likewise, if $\kappa$ is ineffable then $\Diamond_\kappa$ holds, so a fortiori $\Phi_\kappa$ holds.\footnote{A weaker concept of large cardinals that satisfies diamond is a subtle cardinal. But subtle cardinals are not automatically weakly compact, thus the tree property is not guaranteed.}
For the main theorem of this section we redefine the concept of recursive amenability by dropping the requirement that the models in a splendid extension of $M$ have the same cardinality as $M$.

\begin{definition}
  \label{defthetakappa} Let $\kappa$ be a regular cardinal and let $\theta<\kappa$.
  One says that $\kappa$ is $\theta$-recursively amenable if every recursively saturated model $M$ whose cardinality belongs to $[\theta,\kappa)$ has a splendid extension $(M_0,M_1)$ so that $|M|\leq|M_0|,|M_1|<\kappa$.
\end{definition}

Our goal is to show that ineffable cardinals do not satisfy the above definition.
It will follow that models of size less than $\kappa$ with no splendid extension exist, if $\kappa$ is ineffable.

\begin{theorem}
  \label{thmineffable} Let $\kappa$ be an ineffable cardinal and let $\theta<\kappa$.
  Then $\kappa$ is not $\theta$-recursively amenable.
  Hence there exists a model $M$ of size $\lambda\in[\theta,\kappa)$ with no splendid extension.
\end{theorem}

\par\noindent\emph{Proof}. \newline
Since $\kappa$ is ineffable, it is weakly compact and hence the tree property holds at $\kappa$.
From \cite{MR619874} it follows that there are no $\kappa$-Kaufmann models.
Assume towards contradiction that $\kappa$ is $\theta$-recursively amenable.
As in the successor case, for every $\eta\in{}^{<\kappa}2$ we shall define a model $M_\eta$ as follows.
Choose any recursively saturated model $M$ of size $\theta$ and let $M_{\langle\rangle}=M$.
If $\ell{g}(\eta)$ is a limit ordinal, let $M_\eta$ be the union $\bigcup_{\zeta\triangleleft\eta}M_\zeta$.
At the successor step, assuming that $M_\eta$ is already at hand, let $(M_{\eta^\smallfrown 0},M_{\eta^\smallfrown 1})$ be a splendid extension of $M_\eta$.
As usual, we may assume that if $\ell{g}(\eta)=\delta$ is a limit ordinal then $\delta$ is the universe of $M_\eta$.

Recall that $\Phi_\kappa$ holds since $\kappa$ is ineffable.
We define a coloring $c:{}^{<\kappa}2\times{}^{<\kappa}2\rightarrow{2}$ as follows.
Given $\alpha\in\kappa$ and $\zeta,\eta\in{}^\alpha{2}$ we let $X_\zeta$ be the characteristic function of $\zeta$.
Now $c(\zeta,\eta)=\ell$ if $X_\zeta\in{\rm Cod}(M_{\eta^\smallfrown\ell}/M_\eta)-{\rm Def}(M_\eta)$, where $\ell\in\{0,1\}$.
If $c(\zeta,\eta)$ is not defined by the above rule, let $c(\zeta,\eta)=0$.

Pick a weak diamond function $g\in{}^\kappa{2}$.
Let $N=\bigcup_{\alpha\in\kappa}M_{g\upharpoonright\alpha}$.
As in the main theorem of the previous section, $N$ is $\kappa$-like and recursively saturated.
Likewise, $N$ is rather classless and hence $\kappa$-Kaufmann.
But this contradicts the tree property at $\kappa$.
We conclude, therefore, that $\kappa$ is not $\theta$-recursively amenable, as sought.
The additional conclusion of the theorem follows.

\hfill \qedref{thmineffable}

As in the previous section, we actually get something stronger:

\begin{corollary}
  \label{coreverymodel} Let $\kappa$ be an ineffable cardinal.
  If $|M'|<\kappa$ then $M'$ extends to a model $M$ of size less than $\kappa$ with no splendid extension.
\end{corollary}

\hfill \qedref{coreverymodel}

The existence of models of uncountable cardinality with no splendid extension, as proved in the above theorem, is carried-out in \textsf{ZFC} (and the assumption that there is an ineffable cardinal).
However, the proof does not specify the cardinal $\lambda$ on which there is such a model.
The theorem of the previous section is proved by forcing, but it gives such models at prescribed specific cardinals.
One interesting aspect of ineffability is that this large cardinal does not contradict the axiom of constructibility.

\begin{corollary}
  \label{corv=l} If $V=L$ and there exists an ineffable cardinal $\kappa$, then for some $\lambda\in\kappa$ there are models of cardinality $\lambda$ with no splendid extension.
\end{corollary}

\hfill \qedref{corv=l}

From this corollary we deduce that an attempt to prove that models of uncountable size always admit a splendid extension would be somewhat complicated.
In particular, it does not follow immediately from the assumption $V=L$.
In this light, we would like to phrase a few open problems that emerge out of the above results.

\begin{question}
  \label{qstrength} What is the consistency strength of a non-Kaufmann model $M$ of size $\aleph_1$ with no splendid extension?
\end{question}

The consistency strength of the tree property at $\aleph_2$ is well-established, but maybe there is a different strategy with respect to splendid extensions.

\begin{question}
  \label{qglobal} Is it consistent that for every $\kappa=\cf(\kappa)>\aleph_0$ there is a model of size $\kappa$ with no splendid extension simultaneously?
\end{question}

If one follows the strategy of the above proof then the problem resembles Magidor's question about the possible consistency of the tree property at every $\kappa=\cf(\kappa)>\aleph_1$.
But maybe there is a different way to eliminate splendid extensions.

\begin{question}
  \label{qkaufmann} Let $\kappa$ be strongly inaccessible.
  \begin{enumerate}
    \item [$(a)$] Does every model of size $\kappa$ admit a splendid extension of cardinality $\kappa$?
    \item [$(b)$] Does there necessarily exist, in \textsf{ZFC}, a $\kappa^+$-Kaufmann model?
  \end{enumerate}
\end{question}

The connection between the two parts of the problem is clear in the light of the above theorems.
Since $\aleph_0$, the first strongly inaccessible cardinal, demonstrates a positive answer to both parts of the problem, it is reasonable to try the positive direction at any strongly inaccessible cardinal.

Finally, one may wonder whether consistently there are splendid extensions for every model of a given cardinality.
The first uncountable cardinal is typical.

\begin{question}
  \label{qpositive} Is it consistent that every model of cardinality $\aleph_1$ has a splendid extension?
\end{question}

A possible direction is to force a stronger version of the Baire category theorem in which one can intersect $\aleph_1$-many dense sets.
We note, however, that the problem seems to be more complicated and, in particular, one needs more than that.
To see this, consider \textsf{PFA}, under which such a version of the Baire category theorem holds.
By \cite{MR776640}, the tree property holds at $\aleph_2$ under \textsf{PFA}.\footnote{See \cite{MR1174395}.}
Moreover, $2^{\aleph_1}=\aleph_2$ is a consequence of \textsf{PFA}, thus models of size $\aleph_1$ with no splendid extension exist under \textsf{PFA} as follows from Theorem \ref{thmmt}.

\newpage

\bibliographystyle{alpha}
\bibliography{arlist}

\end{document}